\documentclass[11pt,reqno,a4paper]{amsart}
\usepackage{graphicx,color}

\usepackage{amsmath,amsfonts,amssymb,amsthm,mathrsfs,amsopn}
\usepackage{latexsym}

\usepackage[numbers]{natbib}

\usepackage{hyperref}

\usepackage[usenames,dvipsnames,x11names,table]{xcolor}

\numberwithin{equation}{section}

\def\H{\mathcal H}

\def\R{\mathbb R}
\def\N{\mathbb N}
\def\Z{\mathbb Z}
\newcommand{\dist}{\mathop{\mathrm{dist}}}
\def\e{\varepsilon}
\def\s{\sigma}

\def\vphi{\varphi}

\def\Id{{\rm Id}}

\def\spt{{\rm spt}}

\def\pa{\partial}

\def\00{{\bf 0}}

\def\G{\mathbb{G}}

\def\P{\mathcal{P}}

\newcommand{\D}{\Delta}

\newcommand{\tr}{\mbox{tr }}

\newcommand{\cc}{\subset\subset}

\def\Lip{{\rm Lip}\,}

\newcommand\res{\mathop{\hbox{\vrule height 7pt width .3pt depth 0pt \vrule height .3pt width 5pt depth 0pt}}\nolimits}

\def\weak{\stackrel{*}{\rightharpoonup}}

\def\FF{\mathbf{F}}

\def\D{\mathsf{D}}

\def\Id{\mathrm{Id}}

\def\Ker{{\rm Ker}}

\theoremstyle{plain}
\newtheorem{theorem}{Theorem} [section]

\newtheorem{lemma}[theorem]{Lemma}

\theoremstyle{definition}
\newtheorem{definition}[theorem]{Definition}
\newtheorem{remark}[theorem]{Remark}
\newtheorem*{ack}{Acknowledgements}

\title{Existence results for minimizers of parametric elliptic functionals}

\author{Guido De Philippis}
\address{SISSA, Via Bonomea 265, 34136 Trieste, Italy}
\email{gdephili@sissa.it}

\author{Antonio De Rosa}
\address{Institut f\"ur Mathematik, Universit\"at Z\"urich, Winterthurerstrasse 190, CH-8057 Z\"urich, Switzerland}
\email{antonio.derosa@math.uzh.ch}

\author{Francesco Ghiraldin}
\address{Universit\"at Basel, Spiegelgasse 1, CH-4051 Basel,Switzerland. (Formerly Max Planck Institute, Leipzig (DE) )}
\email{francesco.ghiraldin@unibas.ch, Francesco.Ghiraldin@mis.mpg.de}

\begin{document}

\begin{abstract} 

We prove a compactness principle for the anisotropic formulation of the Plateau problem in any codimension, in the same spirit of the previous works of the authors \cite{DelGhiMag,DePDeRGhi,DeLDeRGhi16}. In particular, we perform a new strategy for the proof of the rectifiability of the minimal set, based on the new anisotropic counterpart of the Allard rectifiability theorem proved  in \cite{DePDeRGhi2}. As a consequence we provide a new proof of the Reifenberg existence theorem.


\end{abstract}


\maketitle

\section{Introduction and Main Result}\label{intro}
This paper concludes a series of works by the authors on the Plateau problem: we here provide a general and flexible existence result for sets that 
minimize an anisotropic energy, which can be applied to several notions of boundary conditions.

In the spirit of the previous works \cite{DelGhiMag},\cite{DePDeRGhi} and \cite{DeLDeRGhi16},
we use the direct method of the calculus of variations to find a generalised minimizer  (namely a Radon  measure) via standard compactness arguments, and then we aim at proving that it is actually a fairly regular surface. 
To do this, we employed several techniques to first establish the rectifiability of the limit measure: in the case of the area integrand this property was initially deduced from a powerful
 result due to Preiss \cite{Preiss, DeLellisNOTES}, as well as, in codimension one, from the theory of sets of finite perimeter. 
 These two techniques are no longer available in the case of anisotropic problems in higher codimension (in particular due to the lack of a monotonicity formula for anisotropic problems). A new rectifiability criterion was found in \cite[Theorem 1.2]{DePDeRGhi2}, for varifolds having positive lower density  and a bounded anisotropic first variation, extending the celebrated result by Allard \cite{Allard}, see also \cite{DeRosa}. 

The proof of the existence theorem can be applied to the minimization of the energy in several classes of sets, 
corresponding to several notions of boundary conditions as the one used in \cite{HarrisonPugh15,HarrisonPugh16}. 
In particular we  discuss the existence theorem for minimizers with a homological notion of boundary, originally considered by Reifenberg in the isotropic case \cite{reifenberg1}, see Section~\ref{s:Reif}. Our techniques can as well be extended to prove existence for the anisotropic Plateau problem under co-homological boundary conditions, first considered in~\cite{HarrisonPugh16} where however  more general  assumptions on the integrand are allowed, see  Remark~\ref{rmk:onesto} and~\ref{rmk:cohom} below. 
Recently two related existence theorems have been proved in \cite{Fang2013} and in  \cite{FanKol17}, following the strategy of \cite{Almgren68, Almgren76}. 
In particular in \cite{Fang2013} an existence Theorem in the Reifenberg setting is extended to the case of \(\Z\) coefficients, see the discussion after Theorem \ref{thm:reif} below.

\bigskip
In order to precisely  state our main result, we  introduce some notations and definitions. 
We will always work in \(\R^n\) and \(1\le d\le n \) will 
always be an integer number. 
We recall that a set \(K\) is said to be \(d\)-rectifiable if it can be covered, 
up to an \(\H^d\) negligible set, by countably many \(C^1\) manifolds where \(\H^d\) is the 
\(d\)-dimensional Hausdorff measure.   We also denote by
$G=G(n,d)$ the Grassmannian of unoriented $d$-dimensional linear subspaces of $\R^{n}$ and, for every $U\subset \R^n$, we define $G(U):=U\times G$. Given a $d$-rectifiable set $K$, we denote by $T_K(x)$ the approximate tangent space of $K$ at $x$, 
which exists for $\H^d$-almost every point $x \in K$ \cite[Chapter 3]{SimonLN}. 
We also  let $\mathsf{Lip}(\R^n)$ be the space of Lipschitz maps in $\R^n$.

The anisotropic Lagrangians considered in the rest of the note will be $C^1$ maps 
$$
F: \R^{n}\times G\ni (x,T)\mapsto F(x,T) \in  \R^+,
$$
verifying the lower and upper bounds 
\begin{equation}\label{cost per area}
0 < \lambda \leq F(x,T) \leq \Lambda<\infty.
\end{equation}
Given a $d$-rectifiable set $K$ and an open subset $U\subset \R^n$, we define:
\begin{equation}\label{energia}
\FF(K,U) := \int_{K\cap U} F(x,T_K(x))\, d\H^d(x) \mbox{ \ \ and \ \ } \FF(K) := \FF(K,\R^{n}).
\end{equation}
It will be also useful to look at the ``frozen'' functional: for $y\in\R^{n}$, we let
\begin{equation*}
\FF^y(K,U) := \int_{K\cap U} F(y,T_K(x))\, d\H^d(x).
\end{equation*}
We note that given a  \(d\)-dimensional varifold  \(V\) (i.e., a positive Radon measure on the Grassmannian \(G(U)\)) we can define its anisotropic energy as 
\[
\FF(V,U) := \int F(x,T)\, dV(x,T);
\]
this definition is coherent with~\eqref{energia}, since to any  rectifiable set \(K\) we will naturally associate the varifold \(\H^d\res K\otimes\delta_{T_K(x)}\). In this setting we define the {\em anisotropic first variation} of a varifold \(V\) as the order one distribution whose action on \(g\in C_c^1(U,\R^n)\) is given by 
 \begin{equation*}\label{eq:firstvarintro}
 \begin{split}
 \delta_F V(g): =&\,\frac{d}{dt}\FF\big ( \varphi_t^{\#}V\big)\Big|_{t=0}
 \\
 =&\int_{\Omega\times G(n,d)}  \Big[\langle d_xF(x,T),g(x)\rangle+  B_F(x,T):Dg(x)  \Big] dV(x,T),
\end{split}
 \end{equation*}
 where  \(\varphi_t(x)=x+tg(x)\), \( \varphi_t ^{\#}V\) is the image varifold of \(V\) through \(\varphi_t\) see~\cite[Chapter 8]{SimonLN}, 
 \(B_F(x,T)\in \R^n\otimes\R^n\) is an explicitly computable \(n\times n\) matrix 
  and  \(\langle A,B\rangle :=\tr  A^* B\) for \(A,B\in \R^n\otimes \R^n\), see for instance~\cite{DePDeRGhi2} for  the relevant computations. A varifold \(V\) is said to be \(F\)-stationary in an open set  \(U\) if \(\delta_F V=0\) as a distribution in \(U\).

\medskip

Throughout all the paper, $H\subset \mathbb R^{n}$ will denote a closed subset of $\R^{n}$. 
Assume to have a class $\mathcal{P} (H)$ of relatively closed $d$-rectifiable subsets $K$ of $\mathbb R^{n}\setminus H$. One can then formulate the anisotropic Plateau problem by asking whether the infimum
\begin{equation}
  \label{plateau problem generale}
m_0 :=  \inf \big\{\FF(K) : K\in \mathcal{P}(H)\big \}
\end{equation}
is achieved by some set (which should be a suitable  limit of a minimizing sequence).
A second question is whether this set belongs to the chosen class $\P(H)$ and which additional regularity properties  it satisfies. We will say that a sequence $\{K_j\} \subset \mathcal{P} (H)$ is a {\em minimizing sequence} if   $\FF (K_j) \downarrow m_0$.

We next outline a set of flexible and rather weak requirements for $\mathcal{P}(H)$: the key property 
for $K'$ to be a competitor of $K$ is that $K'$ is close in energy to sets obtained from $K$ via deformation maps as in Definition \ref{d:deform}. 
This allows for a slightly  larger flexibility on the choice of the admissible sets, since a priori $K'$ might not belong to the competition class.

\begin{definition}[Lipschitz deformations]\label{d:deform}
Given a ball $B_{x,r}$, we let $\mathfrak D(x,r)$ be the set of functions $\vphi:\R^n \rightarrow \R^n$ 
such that $\varphi(z)=z$ in $\R^n\setminus B_{x,r}$ and which are smoothly isotopic 
to the identity inside $B_{x,r}$, namely those for which there exists an isotopy $\lambda \in C^\infty([0,1]\times \R^n;\R^n)$ such that 
$$\lambda(0,\cdot) = \Id, \quad \lambda(1,\cdot)=\vphi, \quad \lambda(t,y)=y \quad\forall\,(t,y)\in [0,1]\times (\R^n \setminus B_{x,r}) \quad \mbox{ and } $$
$$ \lambda(t,\cdot) \mbox{\ is a diffeomorphism of } \R^n \ \forall t \in [0,1].
$$
We finally set $\D(x,r):=\overline{\mathfrak D(x,r)}^{w^*-W^{1,\infty}}$, the sequential closure of $\mathfrak D(x,r)$ with respect to the uniform convergence with equibounded differentials.
\end{definition}
Observe that in the definition of $\D(x,r)$ it is equivalent to require any $C^k$ 
regularity on the isotopy $\lambda$, for $k\ge 1$, as $C^k$ isotopies  coinciding with the identity outside $B_{x,r}$ 
can be approximated in $C^k$ by smooth ones also supported in the same set.

\begin{definition}[Deformed competitors and good class]\label{def good class}
Let $H\subset \mathbb R^{n}$ be closed,  $K\subset \mathbb R^{n}\setminus H$ be a relatively closed countably $\H^d$-rectifiable set and  $B_{x,r}\subset \mathbb R^{n}\setminus H$.
A  {\em deformed competitor} for $K$ in $B_{x,r}$ is any set of the form
\begin{equation*}
 \varphi \left (  K \right ) \quad \mbox{ where } \quad \varphi \in \D(x,r).
\end{equation*}

Given a family $\mathcal{P} (H)$ of relatively closed \(d\)-rectifiable subsets $K\subset \mathbb R^{n}\setminus H$, we say that $\mathcal{P} (H)$ is a {\em good class} if for every $K\in\mathcal{P} (H)$, for every $x\in K$ and for a.e. $r\in (0, \dist (x, H))$
\begin{equation}
  \label{inf good class}
  \inf \big\{ \FF (J) : J\in \mathcal{P} (H)\,,J\setminus \overline{B_{x,r}} =K\setminus \overline
{B_{x,r}} \big\} \leq \FF (L)\,
\end{equation}
whenever $L$ is any deformed competitor for $K$ in $B_{x,r}$.
\end{definition}

We will assume the following ellipticity condition on the energy $\FF$, introduced in \cite{Almgren68}, which is a geometric version of quasiconvexity, cf. \cite{Morrey}:

\begin{definition}[Elliptic integrand, {\cite[1.2]{Almgren68}}]\label{d:Almgrenell}
The anisotropic Lagrangian $F$ is said to be \emph{elliptic} if there exists $\Gamma \geq 0$ such that, whenever $x \in \R^{n}$ and $D$ is a $d$-disk centered in $x$, with radius $r$ and contained in a $d$-plane $T$, then the inequality
\begin{equation}\label{eq:ellitico}
\FF^x(K,B_{x,r})- \FF^x(D,B_{x,r}) \geq \Gamma(\H^d(K \cap B_{x,r}) - \H^d(D))
\end{equation}
holds for every compact $d$-rectifiable set $K \subset \overline{B_{x,r}}$ such that $D \subset \pi_T(K)$, where $\pi_T:\R^n \to T$  is the orthogonal projection of $\R^n$ onto $T$.
\end{definition}
\begin{remark}\label{r:lagr}
Given a $d$-rectifiable set $K$ and a deformation $\varphi \in \D(x,r)$, using Property \eqref{cost per area}, we deduce the quasiminimality property
 \begin{equation}\label{quasiminimo}
 \FF(\varphi(K))\leq \Lambda \H^d(\varphi(K)) \leq \Lambda (\Lip(\varphi))^d \H^d(K) \leq \frac{\Lambda}{\lambda} (\Lip(\varphi))^d \FF(K).
 \end{equation}
\end{remark}

In \cite{DePDeRGhi2} the authors obtained an extension of Allard's rectifiability Theorem for  varifolds stationary with respect to anisotropic integrands. Let us first recall the following condition, introduced in \cite{DePDeRGhi2}:

\begin{definition}\label{atomica}
For a given integrand  $F\in C^1(\Omega\times G(n,d)) $, \(x\in \Omega\),  and a Borel  probability measure $\mu \in \mathcal P (G(n,d))$, let us define  
\begin{equation*}\label{eq:A}
A_x(\mu):=\int_{G(n,d)} B_F(x,T )d\mu(T)\in \R^n\otimes \R^n.
\end{equation*}
We say that \(F\) verifies the \emph{atomic condition} $(AC)$ at \(x\) if the following two conditions are satisfied:
\begin{itemize}
\item[(i)]  \(\dim\ker A_x(\mu)\le n-d\) for all \(\mu \in \mathcal P (G(n,d))\),
\item[(ii)]  if  \(\dim\ker A_x(\mu)= n-d\), then  \(\mu=\delta_{T_0}\) for some \(T_0\in G(n,d)\).
\end{itemize}
\end{definition}

The following rectifiability criterion is the main result in  \cite{DePDeRGhi2}:

\begin{theorem}[{\cite[Theorem 1.2]{DePDeRGhi2}}]\label{thm:mainCPAM}
 Let $F\in C^1(G(\R^n),\R^+) $ be a positive integrand satisfying the (AC) condition, and let us suppose that 
 $V$ is a $d$-dimensional varifold such that:
 \begin{itemize}
 \item $V$ has bounded anisotropic first variation: $\delta_F V$ is a Radon measure. 
 \item $V$ has a lower density bound: there exists $\theta_0>0$ such that
 \[
 \frac{\| V \|(B_{x,r})}{r^d}= \frac{ V (B_{x,r}\times G(n,d))}{r^d}\geq\theta_0\qquad\mbox{ for all }x\in K \mbox{ and }r<dist(x,H).
 \]
 \end{itemize}
 Then $V$ is $d$-rectifiable.
\end{theorem}

We can now state  the  main result of this paper:

\begin{theorem}\label{thm generale}
Let \(F\in C^1(G(\R^n))\) be an integrand satisfying~\eqref{cost per area},~\eqref{eq:ellitico} and the \(AC\) condition.
Let $H\subset \mathbb R^{n}$ be closed and $\mathcal{P} (H)$ be a good class. Assume that the infimum in the Plateau problem \eqref{plateau problem generale} is finite and let $\{K_j\}\subset \mathcal{P}(H)$ be a minimizing sequence. Then, up to subsequences, the measures $\mu_j := F(\cdot,T_{K_j}(\cdot))\H^d \res K_j$ converge weakly$^\star$ in $\mathbb R^{n}\setminus H$ 
to the measure $\mu = F(\cdot,T_K(\cdot)) \H^d \res K$, where $K = \spt\, \mu \setminus H$ is a  $d$-rectifiable set. 
Furthermore, the integral varifold naturally associated to $\mu$ is $F$-stationary in $\R^{n}\setminus H$. 
In particular, $\liminf_j\FF(K_j)\ge \FF(K)$ and if $K \in \mathcal{P} (H)$, then $K$ is a minimum for \eqref{plateau problem generale}.
\end{theorem}

\begin{remark}\label{rmk:onesto}
Our proof of Theorem \ref{thm generale} strongly relies on  Theorem \ref{thm:mainCPAM} for which the \(AC\) condition is essentially necessary, see the discussion in \cite{DePDeRGhi2}. In the case when   \(d=n-1\) (or if \(d=1\)) \(AC\) condition is equivalent to the  the strict convexity of $F$ and thus  it implies~\eqref{eq:ellitico} with \(\Gamma=0\), see \cite[Theorem 1.3]{DePDeRGhi2}. When \(2\le d\le (n-1)\) the situation is more subtle and a complete   characterization is not  known.  However, in  the recent work \cite{DK} of the second author with S. Kolasinski  (completed while the current one was submitted) it is shown  that  the  $AC$ condition implies the validity of~\eqref{eq:ellitico} with \(\Gamma=0\), see  \cite[Theorem 8.9]{DK}. Hence to prove Theorem  \ref{thm generale} one can just assume that \(F\) satisfies the \(AC\) condition. 
\end{remark}

\begin{remark}
For the boundary conditions considered in \cite{HarrisonPugh14} and in \cite{davidplateau,DavidBeginners}, the corresponding classes $\mathcal{P} (H)$ are all good, as proved in the previous works \cite{DelGhiMag,DePDeRGhi,DeLDeRGhi16} . By the same arguments in these papers, one can also show that the set  $K$ in Theorem \ref{thm generale} is actually a minimum  for the formulation in \cite{HarrisonPugh14}, see \cite[Theorem 1.5]{DePDeRGhi} (here restricting our deformation to those which can be obtained as limit of orientation preserving diffeomorphism turns out to be crucial).  In this paper  we extend this analysis to provide a new proof of Reifenberg existence theorem,\cite{reifenberg1},    which does not rely on his topological disk theorem, see Section \ref{s:Reif} for the precise setting and Theorem \ref{thm:reif} there.
\end{remark}

\begin{remark}
We observe that in case the set $K$ provided by the Theorem \ref{thm generale} belongs to $\mathcal P(H)$, it has minimal $\FF$ energy with respect to deformations in the classes $\D(x,r)$ of Definition \ref{d:deform}, with $x \in K$ and $H \cap B_{x,r}=\emptyset$.  
Although the union of these classes is strictly contained in the class of all Lipschitz deformations, 
it is rich enough to generate the comparison sets in \cite{Almgren76}, Hence if one assumes that \eqref{eq:ellitico} is true with some \(\Gamma>0\),  one obtains almost everywhere regularity of $K$,   see \cite[III.1 and III.3]{Almgren76}.  
\end{remark}

\begin{ack}
G.D.P. is supported by the MIUR SIR-grant {\it Geometric Variational Problems} (RBSI14RVEZ). A.D.R. is supported by SNF 159403 {\it Regularity questions in geometric measure theory}. 
F.G. is supported by the ERC Starting Grant {\it FLIRT - Fluid Flows and Irregular Transport}. 
\end{ack}

\section{Proof of Theorem \ref{thm generale}}\label{mainresult}
The proof of Theorem \ref{thm generale} shares some arguments with the results in~\cite{DePDeRGhi}: 
in order to keep the outline simple we will sometimes refer to that paper for the common parts, while presenting the original content of the proof in detail.  

We will use the following  notations:  $Q_{x,l}$  denotes the closed cube centered in $x$, 
with edge length $l$;  
\begin{equation*}\label{rettangoli}
R_{x,a,b}:=x+\Big[-\frac a2,\frac a2\Big]^d \times \Big[-\frac b2,\frac b2\Big]^{n-d}
\end{equation*}
denotes the ``rectangle'' with sides \(a\) and \(b\), in the splitting $\R^n = \R^d\times\R^{n-d}$ (hence \(Q_{x,l}=R_{x,l,l}\)). 
 Cubes and balls in the subspace $\R^d\times\{0\}^{n-d}$ are denoted by $Q_{x,l}^d$ and $B_{x,r}^{d}$, respectively.

\begin{proof}[Proof of Theorem \ref{thm generale}]
Since the infimum in the Plateau problem \eqref{plateau problem generale} is finite,  there exists a minimizing sequence $\{K_j\}\subset \mathcal{P}(H)$ and a Radon measure $\mu$ on $\R^n\setminus H$ such that
\begin{equation}\label{muj va a mu}
  \mu_j\weak\mu\,,\qquad\mbox{as Radon measures on $\R^n\setminus H$}\,,
\end{equation}
where $\mu_j=F(\cdot,T_{K_j}(\cdot))\H^d \res K_j$. We set $K = \spt\, \mu \setminus H$ and 
consider also the canonical density-one rectifiable varifolds $V^j$ naturally associated to $K_j$:
\[
V^j :=\mathcal H^d\res K_j \otimes \delta_{T_xK_j}.
\]
Since $K_j$ is a minimizing sequence in \eqref{plateau problem generale} and $F \geq \lambda$, 
 for large $j$ we have the bound\footnote{Here \(\|V\|\) is the projection of the measure  \(V\) on the first factor, i.e. \(\|V\|(A)=V(A\times \G)\) for every Borel set \(A\subset \R^d\)}  
$\|V^j\|(\R^n)\leq \frac{2m_0}{\lambda}$ and therefore we can assume that  $V^j$ converges to $V$ in the sense of varifolds.

\medskip
\noindent
{\em Step 1:  $V$ is $F$-stationary in $\mathbb R^{n}\setminus H$.} 
Assume by contradiction the existence of $g\in C^1_c(\R^n\setminus H,\R^n)$ such that $\delta_FV(g) <0$. 
 This information can be localized to a ball using a partition of unity for the compact set $\spt($g$)\subset\subset\R^n\setminus H$: we can assume that for some  ball $B_{x,r}\cc \R^n\setminus H$ there is a vector field (not relabeled) $g\in C^1_c(B_{x,r},\R^n)$ such that $\delta_FV(g) =: -2c<0$.
There exists $s>0$ such that $(\Id + tg)\in \D(x,r)$ for every $t\leq s$. By continuity of the functional $\delta_F Z(g)$ with respect to $Z$, up to consider a smaller $s>0$, it holds
$$\delta_F (\Id + tg)^{\#}V(g)\leq -c <0, \quad \forall t\in [0, s].$$
Integrating the last inequality in $t$, we conclude that
\begin{equation}\label{integro1}
\FF((\Id + sg)^{\#}V)\leq \FF(V) -cs.
\end{equation}
Moreover there exists an open set $A$ with $B_{x,r}\subset A\subset \subset \R^n \setminus H$ and satisfying 
\begin{equation}\label{11}
\|V\|(\partial A)=\|(\Id + sg)^{\#}V\|(\partial A)=0.
\end{equation}
Since $\Id + sg \in C^1_c(A)$, then \eqref{integro1} reads
\begin{equation}\label{22}
\FF ((\Id + sg)^{\#}V, A) \leq \FF(V,A)-cs  .
\end{equation}
Combining \eqref{11} with \eqref{22}, for $j$ large enough we infer 
\begin{equation}\label{55}
\FF ((\Id + sg)^{\#}V^j, A) \leq \FF(V^j,A) -cs  +o_j, \qquad \mbox{where $\lim_{j\to \infty} o_j=0$},
\end{equation}
see also \cite[Section 5.1]{DeRosa} for more details. Note that $\FF ((\Id + sg)^{\#}V^j, A) =\FF ((\Id + sg)(K_j), A)$ as well as $\FF(V^j,A)= \FF (K_j, A)$.
Adding $\FF(K_j,\R^n\setminus A)$ to both sides of \eqref{55} and observing that $(\Id + sg)(K_j)\setminus A =K_j\setminus A$, we obtain
\begin{equation}\label{33}
\FF ((\Id + sg)(K_j), \R^n) \leq \FF(K_j,\R^n) -cs  + o_j(1).
\end{equation}
Since $\Id + sg \in \D(x,r)$ and $B_{x,r}\cc \R^n\setminus H$, by definition of good class, see Definition \ref{def good class}, there exists a new sequence $\{\tilde K_j\}_{j\in \N}\subseteq \mathcal{P} (H)$, such that
\begin{equation}\label{44}
\FF(\tilde K_j)\leq \FF((\Id + sg)(K_j)) + \frac{cs}{2}, \qquad \forall j \in \N.
\end{equation}
Combining \eqref{33} with \eqref{44} and choosing $j$ big enough, we contradict the minimizing property of  the sequence $\{K_j\}_{j\in \N}$. 

\medskip
\noindent
{\em Step 2:  $V$ satisfies a density lower bound}.  We claim that there exists  $\theta_0=\theta_0(n,d,\lambda,\Lambda)>0$ such that
\begin{equation}
  \label{lower density estimate mu ball}
 \|V\|(B_{x,r})\ge\theta_0\,\omega_d r^d\,,\qquad \textrm{ \(x\in\spt\,\|V\| \) and \( r<d_x :=\dist(x,H)\)}.
\end{equation}
This can be achieved by the same techniques of~\cite[Theorem 1.3, Step 1]{DePDeRGhi}. Indeed  by~\eqref{cost per area} the integrand \(F\) is comparable to the area and this is the only property needed in the proof in~\cite{DePDeRGhi}, see also~\cite{DeRosa}.

\medskip
\noindent

{\em Step 3: \(V\) is rectifiable}. Combining the lower bound~\eqref{lower density estimate mu ball} with the $F$-stationarity in $\R^n\setminus H$ and applying 
Theorem~\ref{thm:mainCPAM}  
we conclude that $V$ is a $d$-rectifiable varifold and in turn that $\mu = \FF(V,\cdot )=\theta \H^d \res \tilde{K}$ for some countably $\H^d$-rectifiable set $\tilde{K}$ and some positive Borel function $\theta$.
Since $K$ is the support of $\mu$, then $\H^d (\tilde{K}\setminus K) =0$. 
On the other hand, by differentiation of Hausdorff measures,  \eqref{lower density estimate mu ball} yields $\H^d (K\setminus \tilde{K}) =0$. Hence $K$ is  \(d\)-rectifiable and 
\begin{equation}\label{supporto}
\mu=\theta \H^d\res K.
\end{equation}

We now proceed to compute the exact value of the density $\theta$: 
to this end we need the following elementary lemma, whose proof can be obtained as in~\cite[Lemma 3.2]{DeLDeRGhi16}.

\begin{lemma}\label{l:lista}
Let $K$ be the $d$-rectifiable set obtained in Step 3. 
For every $x$ where $K$ has an approximate tangent plane $T_K (x)$, let $O_x$ be a special orthogonal transformation of $\mathbb R^{n}$ mapping $\{x_{d+1}=\dots=x_n=0\}$ onto $T_K (x)$ and set
$\bar{Q}_{x,r} = O_x (Q_{x,r})$ and $\bar R_{x,r,\e r} = O_x (R_{x, r , \e r})$. 
At almost every $x\in K$ the following holds: for every $\e>0$ there exist $r_0=r_0(x) \leq\tfrac{1}{\sqrt{n+1}} \dist(x,H)$ such that, for $r\leq r_0 /2$, 
\begin{gather}
(\theta_0 \omega_d-\e)r^d\leq \mu(B_{x,r})\le (\theta(x)\omega_d+\e)r^d,\qquad
(\theta(x)-\e)r^d< \mu(\bar{Q}_{x,r})< (\theta(x)+\e)r^d, \label{uno}\\
\sup_{y\in B_{x,r_0}, \,S\in G}|F(y,S) - F(x,S)|\leq \e,\label{tre}
\end{gather}
where $\theta_0=\theta_0(n,d,\lambda,\Lambda)$ is the universal lower bound obtained in \eqref{lower density estimate mu ball}.
Moreover, for almost every such $r$, there exists  $j_0(r) \in \N$ such that for every $j\geq j_0$:
\begin{gather}
(\theta(x)\omega_d-\e)r^d \le\FF(K_j,B_{x,r})\leq (\theta(x) \omega_d + \e ) r^d,\label{due}\\
(\theta(x)-\e)r^d\leq \FF(K_j,\bar{Q}_{x,r})\leq (\theta(x)+\e)r^d, \qquad \FF(K_j,\bar{Q}_{x,r}\setminus \bar{R}_{x,r,\e r})<\e r^d \label{quattro}.
\end{gather}
\end{lemma}

We are now ready to complete the proof of Theorem \ref{thm generale}, namely to show $\liminf_j\FF(K_j)\ge \FF(K)$ and that  \(\mu=F(x,T_K(x))\H^d \res K\).\\

\noindent
{\em Step 4: $\theta(x) \geq F(x,T_K(x))$  for almost every $x\in K$}. 
We prove that $\theta(x)\ge F(x,T_K)$ for every $x\in K$ that satisfies the properties of Lemma~\ref{l:lista} (thus, $\H^d$-a.e. on $K$). 
Let us fix $\e<r_0/2$ and choose a radius $r<r_0/2$ such that both $r$ and $(1-\sqrt{\e})r$ 
satisfy properties \eqref{uno}-\eqref{quattro} and \( \mu(\pa Q_{r}) = 0\).
Let us assume without loss of generality that \(x=0\) and $T_K = \{x_{d+1}= ... = x_n =0\}$. 
Since the \(\mu_j\) are weakly converging to \(\mu\), we get that for \(j\ge j(r)\)
\begin{equation}\label{eq:10}
 \mu_j(Q_{r} \setminus R_{r,\e r}) =o_j(1). 
\end{equation}

First of all we deform $K_j$ so that the mass in \eqref{eq:10} is zero: 
by repeating verbatim the argument in \cite[Theorem 1.3, Step 4]{DePDeRGhi}, we construct a map
\(\psi \in \D(0,2r)\) such that, setting $\widetilde{K_j}:=\psi(K_j)$, we have (see \cite[Equation (3.16)]{DePDeRGhi})
\begin{equation}\label{em}
\H^d(\widetilde{K_j}\cap (Q_{r(1-\sqrt{\e})}\setminus R_{r(1-\sqrt{\e}),6\e r}))=0,
\end{equation} 
and consequently (see \cite[Equation (3.18)]{DePDeRGhi})
\begin{equation}
\begin{split}
\FF\left (K_j,B_{r(1-\sqrt{\e})}\right )&\geq \FF(\widetilde{K_j},B_{r(1-\sqrt{\e})}) -C\sqrt{\e}r^d. \label{ciaone}
\end{split}
\end{equation}
We claim that 
\begin{equation}\label{pippo}
B_{r(1-\sqrt{\e})} \cap T_K \subset \pi(\widetilde{K_j}\cap R_{r(1-\sqrt{\e}), 6 \e r})
\end{equation}
 where $\pi:\R^n \to T_K$  is the orthogonal projection of $\R^n$ onto $T_K$. Note that this immediately implies that
 \begin{equation*}
\overline{B_{r(1-\sqrt{\e})}} \cap T_K \subset \pi(\widetilde{K_j}\cap \overline{B_{r(1-\sqrt{\e})}}).
\end{equation*}
 Therefore the ellipticity of $F$ (namely \eqref{eq:ellitico}), \eqref{tre}, \eqref{due}, \eqref{quasiminimo} and \eqref{lower density estimate mu ball} imply that
\begin{equation}
\label{cip}
\FF^0(T_K,B_{r(1-\sqrt{\e})})\leq \FF^0(\widetilde{K_j},B_{r(1-\sqrt{\e})})\leq \FF(\widetilde{K_j},B_{r(1-\sqrt{\e})}) + C\e r^d,
\end{equation}
hence
\begin{equation*}
\begin{split}
\theta(0)\omega_dr^d(1-\sqrt{\e})^d& \overset{\eqref{due}}{\geq}  \FF(K_j,B_{r(1-\sqrt{\e})}) -\e r^d \overset{\eqref{ciaone}}{\geq} \FF(\widetilde{K_j},B_{r(1-\sqrt{\e})}) -C\sqrt{\e}r^d \\
& \overset{\eqref{cip}}{\geq}\FF^0(T_K,B_{r(1-\sqrt{\e})})-C\sqrt{\e}r^d 
 = F(0,T_K)\omega_d r^d(1-\sqrt{\e})^d -C\sqrt{\e}r^d 
\end{split}
\end{equation*}
which yields $\theta(0)\geq F(0,T_K)$.

We are thus left to prove \eqref{pippo}. To this end, as in \cite[Theorem 1.3, Step 4]{DePDeRGhi}, we extend $\pi_{|R_{r(1-\sqrt{\e}),6 \e r}}$ to a map $P \in \D(0,2r)$ that collapses $R_{r(1-\sqrt{\e}),6\e r}$  onto $T_K$, is the identity outside $Q_{r}$ and satisfies \({\rm Lip}  \,P\le 1+C\sqrt{\e}\) for some dimensional constant \(C\).  

We now set $K_j':=P(\widetilde{K_j})=(P\circ \psi) (K_j)$. Thanks to \eqref{em},
\begin{equation}\label{empty}
\H^d\Big(K_j'\cap \big(Q_{(1-\sqrt{\e})r}\setminus Q_{(1-\sqrt{\e})r}^d\big)\Big)=0;
\end{equation}
furthermore the
estimates  \eqref{quattro} and \eqref{quasiminimo} imply also
\begin{equation}\label{franco}
\begin{split}
\FF\left (K'_j,Q_{r}\setminus Q_{r(1-\sqrt{\e})}\right )&\leq C(\Lambda, \lambda, P) \FF\left (K_j,Q_{r}\setminus Q_{r(1-\sqrt{\e})}\right )\leq C\sqrt{\e} r^d. 
\end{split}
\end{equation}

Assume now by contradiction that \eqref{pippo} fails. 
Since $K_j'$ is closed and $P = \pi$ in $R_{r(1-\sqrt{\e}),6\e r}$, there would exist $y'_j \in Q_{(1-\sqrt{\e})r}^d$ and $\delta_j >0$ such that, if we set $y_j:=(y'_j, 0)$, then
\begin{equation}\label{eq:11}
K_j' \cap B^{d}_{y_j,\delta_j} = \emptyset \quad \mbox{and} \quad B^{d}_{y_j,\delta_j}\subset Q_{(1-\sqrt{\e})r}^d.
\end{equation}
We now consider the  deformation map $\varphi_j $ 
retracting $Q_{(1-\sqrt{\e})r}^d \setminus B^{d}_{y_j,\delta_j}$ onto the equator $\partial Q_{(1-\sqrt{\e})r}^d$ and not 
moving $Q_{(1-\sqrt{\e})r}^c$,    
constructed in   \cite[Theorem 1.3, Step 4]{DePDeRGhi}. In particular,  letting  $K''_j:=\varphi_j(K_j')$,  the following properties are satisfied: 
\[
\varphi_j(Q_{(1-\sqrt{\e})r}^d) \subset \pa Q_{(1-\sqrt{\e})r}^d
\]
and
\[
K''_j \setminus Q_{r} = K_j \setminus Q_{r}.
\]
Thanks to \eqref{empty} we obtain:  
\begin{equation}\label{:)}
\H^d(K_j''\cap Q_{r(1-\sqrt{\e})})=0.
\end{equation}
By the very definition of  $\P(H)$ we can find  $J_j\in \P(H)$, $\e r^d$-close in energy to $K''_j$. 
Since $K'_j$ and $K''_j$ coincide outside $Q_{(1-\sqrt{\e})r}$, we then conclude that:
\begin{equation*}
\begin{split}
\FF(K_j)-\FF(J_j) &\, \,\geq \FF\left (K_j,Q_r\right )-\FF\left (K''_j,Q_r\right ) - \e r^d \\ 
&\overset{\eqref{:)}}{\geq}  \FF(K_j,Q_{r(1-\sqrt{\e})}) + \FF\left (K_j,Q_{r}\setminus Q_{r(1-\sqrt{\e})}\right ) -\FF\left (K'_j,Q_{r}\setminus Q_{r(1-\sqrt{\e})}\right ) - \e r^d \\ 
&\overset{\eqref{franco}}{\geq} (\theta(0)-\e)(1-\sqrt{\e})^dr^d- C\sqrt{\e} r^d\geq (\theta(0) - C\sqrt{\e})r^d >0 .
\end{split}
\end{equation*}
Since this is in contradiction with the minimizing property of the sequence $\{K_j\}$, we conclude the proof of \eqref{pippo}.

\medskip

\noindent
{\em Step 5: $\theta(x)\leq F(x,T_K(x))$ for almost every $x\in K$}: Again we assume that \(x=0\) is a point  satisfying  the conclusion  of Lemma~\ref{l:lista} and we argue as in  \cite[Theorem 1.3, Step 5]{DePDeRGhi}.

Arguing by contradiction, we assume that $\theta(0)=F(0,T_K(0))+\sigma$ for some  $\sigma>0$ and let $\e<\min\{\tfrac{\sigma}{2},\frac{\lambda\s}{ 4\Lambda}\}$. 
As a consequence of \eqref{quattro}, there exist $r$ and $j_0=j_0(r)$ such that
  \begin{equation}
    \label{buonpunto}
      \FF(K_j,Q_{r})>\Big(F(0,T)+\frac{\s}2\Big)\,r^d\,,\qquad \FF(K_j,Q_{r}\setminus R_{r,\e r})<\frac{\lambda\s}{ 4\Lambda}\,r^d,\qquad\forall j\ge j_0\,.
  \end{equation}
 Consider the map $P\in \D(0,r)$  defined in \cite[Equation 3.14]{DePDeRGhi} which collapses \(R_{r(1-\sqrt{\e}),\e r}\) onto the tangent plane $T_K$ and satisfies  $\|P - Id\|_\infty +\Lip (P - Id)\le C\,\sqrt\e$. Exploiting the fact that \(\mathcal P(H)\) is a good class and by almost minimality of $K_j$, we find that
  \begin{equation*}
  \begin{split}
    \FF(K_j,Q_{r})-o_j(1) &\le 
    \underbrace{\FF( P(K_j),P(R_{(1-\sqrt{\e})r,\e r}))}_{I_1}
    +
    \underbrace{\FF( P(K_j),P(R_{r,\e r}\setminus R_{(1-\sqrt{\e})r,\e r}))}_{I_2} \\
    &+
    \underbrace{\FF( P(K_j),P(Q_{r}\setminus R_{r,\e r}))}_{I_3}\,.
    \end{split}
  \end{equation*}
  By the properties of $P$ and \eqref{tre}, we get $I_1\le (F(0,T_K)+\e) r^d$, while, by \eqref{buonpunto} and equation \eqref{quasiminimo}
  \[
  I_3\leq \frac{\Lambda}{\lambda}\, (\Lip P)^d\,  \FF(K_j,Q_{r}\setminus R_{r,\e r})<(1+C\,\sqrt\e)^d\,\frac{\s}{4 }\,r^d\,.
  \]
  Since $\FF( P(K_j),P(R_{r,\e r}\setminus R_{(1-\sqrt{\e})r,\e r}))\leq \frac{\Lambda}{\lambda}(1+C\sqrt{\e})^d\FF(K_j,R_{r,\e r}\setminus R_{(1-\sqrt{\e})r,\e r})$ and $ R_{r,\e r}\setminus R_{(1-\sqrt{\e})r,\e r}\subset Q_{(1-\sqrt\e)r}\setminus  Q_{r}$, by \eqref{quattro} we can also bound
 \begin{multline*}
  I_2=\FF( P(K_j),P(R_{r,\e r}\setminus R_{(1-\sqrt{\e})r,\e r}))  \le\frac{\Lambda}{\lambda}(1+C\sqrt{\e})^d\FF(K_j,Q_{r}\setminus Q_{(1-\sqrt\e)r})
    \\
    \le C(1+C\sqrt{\e})^d\Big((F(0,T_K)+\s+\e)-(F(0,T_K)+\s-\e)(1-\sqrt\e)^d\Big)\,r^d 
    \le C\sqrt{\e}r^d.
\end{multline*}

Hence, as $j\to\infty$, by \eqref{uno}
  \[
  \Big(F(0,T_K)+\frac{\s}2\Big) r^d\le (F(0,T_K)+\e) r^d +C\sqrt{\e} r^d+(1+C\,\sqrt\e)^d\,\frac{\s}{4}\,r^d\,:
  \]
 dividing by $r^d$ and letting $\e\downarrow 0$ provides the desired contradiction. 
 
 \bigskip
 Combining \eqref{supporto} with Step 4 and Step 5, we conclude the proof of the theorem.
\end{proof}

\section{Solution of the Plateau problem in the Reifenberg's formulation}\label{s:Reif}
The Reifenberg formulation of the Plateau problem in \cite{reifenberg1} 
involves an algebraic notion of boundary described in terms of \v{C}ech homology groups. 
The particular choice of an homology theory defined on compact spaces and with coefficient groups that are abelian and compact 
has three  motivations: 
\begin{itemize}
\item[(i)] These assumptions imply that the homology groups are continuous with respect to the Hausdorff convergence of the spaces.
\item[(ii)] It satisfies the classical axioms of Eilenberg and Steenrod, enabling the use of the Mayer-Vietoris exact sequence, \cite[Chapter X]{ES}.
\item[(iii)] It ensures the crucial property
$$\H^{\ell}(K)=0\Rightarrow \check{H}_{\ell}(K)=0,$$
which may fail in general for other homology theories, \cite{milnor62}.
\end{itemize}
 In this section we follow Reifenberg's approach and show that we can obtain a minimizing set in the chosen homology class.  
 
 \medskip

Let $\G$ be a compact Abelian group and let $K$ be a closed set in $\R^n$. 
 For every $m\geq 0$ we denote with $\check{H}_m(K;\G)$ (often omitting the explicit mention of the group $\G$)
 the $m^\text{th}$-\v{C}ech homology group of $K$ with coefficients in $\G$, \cite[Chapter IX]{ES}.

Recall that, if $H \subset K$ is a compact set, the inclusion map $i_{H,K}: H \to K$
induces a graded homomorphism between the homology groups (of every grade $m$, again often omitted)
$$ i_{*\,H,K} :\check{H}_{m} (H,\G)\to \check{H}_{m} (K,\G).$$ 
Note, in the next definition, the role of the dimension $d$, inherent of our variational problem.

\begin{definition}[Boundary in the sense of Reifenberg]\label{Reif}
Let $\G,H,K$ be as above and let $L \subset \check{H}_{d-1} (H,\G)$ be a subgroup. 
We say that $K$ has boundary $L$ if 
\begin{equation}\label{eq:ker}
\Ker (i_{*\,H,K} )\supset L.
\end{equation}
\end{definition} 

\begin{definition}[Reifenberg class]\label{Reifclass}
Let $\G,H,K,L$ be as above; we let $\mathcal{R} (H)$ be the 
class of closed $d$-rectifiable subsets $K$ of $\R^n \setminus H$ contained in the same ball $B\Supset H$ and
such that $K\cup H$ has boundary $L$ in the sense of Definition \ref{Reif}.
\end{definition} 
\begin{remark}\label{remReif}
We remark that $\mathcal{R} (H)$ is a good class in the sense of Definition \ref{def good class}. 
Indeed for every $K \in \mathcal{R} (H)$, 
every $x\in K$, $r\in (0, \dist (x, H))$ and $\varphi \in \D(x,r)$ (which is in particular continuous), 
$$\varphi (K\cup H)=\varphi (K)\cup H, $$ 
and moreover by functoriality $\Ker(\varphi\circ i_{H,K\cup H})_* \supset \Ker (i_{*\,H,K\cup H}) \supset L$, 
which implies that $\varphi(K)\in \mathcal{R} (H)$.
We can therefore apply Theorem \ref{thm generale} to the class $\mathcal{R} (H)$ to  obtain the existence of a relatively closed subset $K$ of $\R^n\setminus H$ satisfying
\[
\mathbf F(K) = \inf_{S\in \mathcal{R} (H)}\{ \mathbf F(S) \}.
\]
\end{remark}

We address now the question of whether 
$K$ belongs to the Reifenberg class $\mathcal{R} (H)$.
Recall the definition of Hausdorff distance between two compact sets $C_1,C_2 \subset \R^n$:
\[
d_{\mathcal H} (C_1,C_2) :=\inf \{r>0 : C_1\subset U_r(C_2) \mbox{ and } C_2\subset U_r(C_1)\},
\]
where we denote with $U_r(C)$ the $r$-neighborhood of $C$.

The main result of this section is the following theorem:


\begin{theorem}\label{thm:reif}
For every minimizing sequence $\{K_j\}\subset \mathcal{R}(H)$
the associated limit set given by Theorem \ref{thm generale} satisfies $K \in \mathcal{R}(H)$.
\end{theorem}

The proof of the above theorem will be obtained   by constructing  another minimizing sequence, $(K_j^1)\subset \mathcal{R}(H)$, yielding the same limit set $K$, but with the further property that 
\begin{equation}\label{thesis2}
d_\H(K_j^1\cup H,K\cup H)\to 0.
 \end{equation}
 To do this, we will first deform the part of  \(K_k\)  lying  outside an \(\e_j\)-neighbourhood \(U_{\e_j}(K)\) (\(\e_j \to 0\), $k\gg j$) onto a \((d-1)\)-skeleton and then take its intersection with \(U_{\e_j}(K)\). An application of the Mayer-Vietoris sequence will yield  that the new sequence still satisfies the homological boundary conditions, see Step 4 below. 
 
 In turn, this last step   relies on the properties (ii) and  (iii) above. As in the case of the work by Reifenberg \cite{reifenberg1},  this forces us to restrict ourselves to compact coefficient groups. Let us remark that the extension of Reifenberg  theorem to \(\mathbb Z\) coefficients has been given in \cite{Fang2013} with different  techniques.

\begin{proof}
\emph{Step 1: Construction of the new sequence.} From Theorem \ref{thm generale} 
we know that $\mu_k := F(\cdot,T_{(\cdot)}{K_k})\H^d \res K_k$ converge weakly$^\star$ in $\mathbb R^{n}\setminus H$ 
to the measure $\mu = F(\cdot,T_{(\cdot)}K) \H^d \res K$. Then, for every $\e_j>0$, there exists $k_j$ big enough so that
\begin{equation}\label{eq:smallmass}
\mu_{k_j}(\R^n \setminus U_{\e_j}(K\cup H))<\frac {\Lambda \e_j^d }{ k_1(4n)^d},
\end{equation}
 with $\Lambda$ the constant in equation \eqref{cost per area} and $k_1$ is the constant appearing in the Deformation Theorem, \cite[Theorem 2.4]{DePDeRGhi}.

We cover $U_{5\e_j}(K\cup H)\setminus U_{2\e_j}(K\cup H)$ with a complex $\Delta$ of closed cubes with side length equal to $\e_j/(4n)$ contained in $U_{6\e_j}(K\cup H)\setminus U_{\e_j}(K\cup H)$. 
We can apply an adaptation of the Deformation Theorem \cite[Theorem 2.4]{DePDeRGhi} relative to the set $K_{k_j}$ and 
obtain a Lipschitz deformation $\varphi_j:=\varphi_{\e_j/4n,K_{k_j}}$. 
Observe that $\varphi_j(K_{k_j}) \cap (U_{4\e_j}(K\cup H)\setminus U_{3\e_j}(K\cup H))\subset \Delta_d$, the $d$-skeleton of the complex. We claim that  
\begin{equation}\label{d-1}
\varphi_j(K_{k_j}) \cap (U_{4\e_j}(K\cup H)\setminus U_{3\e_j}(K\cup H))\subset \Delta_{d-1}.
\end{equation}  
 Otherwise by~\cite[Theorem 2.4]{DePDeRGhi} point (5),  $\varphi_j(K_{k_j}) \cap (U_{4\e_j}(K\cup H)\setminus U_{3\e_j}(K\cup H))$ should contain an entire $d$-face of edge length $\e_j/4n$, leading together with \eqref{eq:smallmass} to a contradiction:
\begin{equation*}
\begin{split}
\frac {\e_j^d} {(4n)^d}&\leq \H^d(\varphi_j(K_{k_j}) \cap (U_{4\e_j}(K\cup H)\setminus U_{3\e_j}(K\cup H))) \leq  k_1 \H^d(K_{k_j}\setminus U_{\e_j}(K\cup H))\\
&\leq \frac{ k_1}{\lambda} \FF(K_{k_j}\setminus U_{\e_j}(K\cup H)) \leq   \frac{ k_1}{\lambda}\mu_{k_j}(\R^n \setminus U_{\e_j}(K\cup H))<\frac {\e_j^d} {(4n)^d}.
\end{split}
\end{equation*}
Set $\widetilde{K_j}:=\varphi_j(K_{k_j})$: by \eqref{d-1} and the coarea formula \cite[3.2.22(3)]{FedererBOOK}, there exists $\alpha_j \in (3,4)$ such that
\begin{equation}\label{cond1}
\H^{d-1}(\widetilde{K_j}\cap \partial U_{\alpha_j \e_j}(K\cup H))=0.
\end{equation}
We now let 
\begin{equation}\label{cond2}
K_j^1:=\widetilde{K_j}\cap \overline{ U_{\alpha_j \e_j}(K\cup H)} \quad \mbox{and} \quad K_j^2:=\widetilde{K_j}\setminus U_{\alpha_j \e_j}(K\cup H).
\end{equation}
\medskip
\noindent
\emph{Step 2:  proof of the property \eqref{thesis2}.}

Note that by construction,  
\begin{equation}\label{mezzahaus1}
\quad K_j^1\cup H \subset U_{4 \e_j}(K\cup H), \quad \forall j \in \N.
\end{equation}
If on the other hand for some positive $\bar \e$ there is $x\in K\cup H \setminus U_{\bar \e}(K^1_{j}\cup H)$, then 
necessarily $d(x,H)\geq \bar \e$ as well as $d(x,K^1_{j})\geq \bar \e$: the weak convergence $\mu_{k_j}\weak\mu$ would imply that
$$\mu(B(x,\bar \e/2))\leq \liminf_{j \to \infty} \mu_{k_j}(B(x,\bar \e/2))=0,$$
which contradicts the uniform density lower bounds~\eqref{lower density estimate mu ball} on $B(x, \bar \e/2)$. 
This yields~\eqref{thesis2}.

\medskip
\noindent
\emph{Step 3: boundary constraint of the new sequence.}
To conclude the proof of Theorem \ref{thm:reif}, we need to check that $(K_j^1)\subset \mathcal{R}(H)$. 
By \eqref{cond2} we get
$$\widetilde{K_j}=K_j^1\cup K_j^2, \qquad \mbox{and} \qquad K_j^1\cap K_j^2=\widetilde{K_j}\cap \partial U_{\alpha_j \e_j}(K)$$
and \eqref{cond1},\eqref{cond2} yield 
$$\H^{d-1}((K_j^1\cup H)\cap K_j^2)=\H^{d-1}(K_j^1\cap K_j^2)=0.$$
 Therefore by \cite[Theorem VIII 3']{HurewiczWallman}:
\begin{equation}\label{homology}
\check{H}_{d-1}((K_j^1\cup H)\cap K_j^2)=(0).
\end{equation}
We furthermore observe that 
the sets $\widetilde{K_j}$ are obtained as deformations via Lipschitz maps strongly approximable 
via isotopies, and therefore belong to $\mathcal{R}(H)$. 
Since the map $\varphi_j$ coincides with the identity on $H$, we have
\[
 i_{H,\widetilde{K_j}\cup H} =\varphi_{j}\circ i_{H,K_j\cup H};
\]
moreover, trivially  $i_{H,\widetilde{K_j}\cup H} =i_{K_j^1\cup H,\widetilde{K_j}\cup H}\circ i_{ H,K_j^1\cup H} $. 
Hence by functoriality  
$$
\Ker(i_{*\,K_j^1\cup H,\widetilde{K_j}\cup H}\circ i_{*\, H,K_j^1\cup H}) 
= \Ker(i_{*\,H,\widetilde{K_j}\cup H})=\Ker((\varphi_{j})_*\circ i_{*\,H,K_j\cup H})\supset L.
$$
We claim that $i_{*\,K_j^1\cup H,\widetilde{K_j}\cup H}$ is injective: 
this implies that  
\begin{equation}\label{KjinR} 
\Ker(i_{*\, H,K_j^1\cup H})\supset L,
\end{equation}
 namely $(K_j^1)\subset \mathcal{R}(H)$. 
 
\medskip
\noindent
{\em Step 4: the map $i_{*\,K_j^1\cup H,\widetilde{K_j}\cup H}$ is injective.}
We can write the Mayer-Vietoris sequence  (which for the \v{C}ech homology holds true for compact spaces and with coefficients in a compact group, due to the necessity of having the excision axiom, \cite[Theorem 7.6 p.248]{ES}) and use \eqref{homology}:
\begin{equation*}\label{eq:MW}
(0) \overset{\eqref{homology}}{=}\check{H}_{d-1}((K_j^1\cup H)\cap K_j^2)
\overset{f}{\longrightarrow} \check{H}_{d-1}(K_j^1\cup H)\oplus \check{H}_{d-1}( K_j^2 )
 \overset{g}{\longrightarrow} 
\check{H}_{d-1}(\widetilde{K_j}\cup H)
\longrightarrow ...
\end{equation*}
where $f = (i_{*\,(K_j^1\cup H)\cap K_j^2,K_j^1\cup H},i_{*\,(K_j^1\cup H)\cap K_j^2,K_j^2})$ and 
$g (\sigma,\tau)=\sigma-\tau $. 
The exactness of the sequence  
implies that $g$ is injective: 
in particular  the map $g$ is injective when restricted to the subgroup  $\check{H}_{d-1}(K_j^1\cup H)\oplus (0)$, 
where it coincides with $i_{*\,K_j^1\cup H,\widetilde{K_j}\cup H}$. 

\medskip
\noindent
{\em Step 5: boundary constraint for the limit set.}
Setting 
\[
Y_n :=\overline{\bigcup_{j\geq n}K_j^1\cup H},
\] 
by \eqref{thesis2} we get 
\begin{equation}\label{convhaus} 
d_{\mathcal H}(Y_n, K\cup H)\rightarrow 0.
\end{equation}
 Therefore $K\cup H$ is the inverse limit of the sequence $Y_n$. 
Since the sets $(K_j^1\cup H)$ are in the Reifenberg class $\mathcal{R}(H)$, namely the inclusion \eqref{KjinR} holds, 
by composing the two injections $i_{*\,K_j^1\cup H, Y_n}$ and $i_{*\,H,K_j^1\cup H}$ for every $j\geq n$,
we obtain that 
\[
L\subset\Ker(i_{*\,H,Y_n}).
\]
Since the \v{C}ech homology with coefficients in compact groups is continuous \cite[Definition 2.3]{ES}, the latter inclusion is stable under Hausdorff convergence, \cite[Theorem 3.1]{ES} (see also \cite[Lemma 21A]{reifenberg1}):
 therefore, by \eqref{convhaus}, we conclude
 $$
 L\subset \mbox{Ker}(i_{*\,H,K\cup H}),
 $$
  and eventually  $K \in \mathcal{R}(H)$.
\end{proof}

\begin{remark}\label{rmk:cohom}
Using the contravariance of cohomology theory, the same results can be obtained when considering 
a cohomological definition of boundary, again in the \v{C}ech theory, as introduced in \cite{HarrisonPugh16}. 
In particular a new proof the the theorem there can be obtained with our assumption on the Lagrangian.  

Note that in the cohomological definition of boundary all the Eilenberg-Steenrod axioms are satisfied even with a non-compact group $\G$. This allows us to consider as coefficients set the natural group $\Z$.
\end{remark}

\begin{remark}
We observe that any minimizer $K$ as in Theorem \ref{thm:reif} is also an $(\FF,0,\infty)$ minimal set in the sense of \cite[Definition III.1]{Almgren76}. Indeed the boundary condition introduced in Definition \ref{Reif} is preserved under Lipschitz maps (not necessarily in $\D(x,r)$). In particular, by  \cite[Theorem III.3(7)]{Almgren76}, if $F$ is smooth and strictly elliptic ($\Gamma$ in Definition \ref{d:Almgrenell} is strictly positive), then $K$ is smooth away from the boundary, outside of a relative closed set of $\H^d$-measure zero.  
\end{remark}

\end{document}